\documentclass{article}

\usepackage{latexsym,enumerate}
\usepackage{amsmath,amsthm,amsopn,amstext,amscd,amsfonts,amssymb}
\usepackage[ansinew]{inputenc}
\usepackage{verbatim}
\usepackage{pstricks}
\usepackage{wasysym}
\usepackage[all]{xy}
\usepackage{epsfig} 
\usepackage{graphicx,psfrag} 
\usepackage{subfigure}

\newtheorem{theorem}{Theorem}[section]

\newtheorem{proposition}[theorem]{Proposition}
\newtheorem{corollary}[theorem]{Corollary}

\newtheorem{remark}[theorem]{Remark}
\newtheorem{problem}[theorem]{Problem}
\newtheorem{algorithm}[theorem]{Algorithm}

\DeclareMathOperator{\Deg}{Deg}

\newcommand{\spb}[1]{\smallskip}
\newcommand{\mpb}[1]{\medskip}
\newcommand{\bpb}[1]{\bigskip}

\renewcommand{\d}{\delta}

\newcommand{\D}{\Delta}
\newcommand{\G}{\Gamma}

\newcommand{\s}{\sigma}

\begin{document}

\title{Computing the alliance polynomial of a graph} 

%
%
%
%
%
%
%

\author{Walter Carballosa Torres \qquad Jos\'e Manuel Rodr{\'\i}guez Garc\'ia\\
Departamento de Matem\'aticas, Universidad Carlos III de Madrid,\\
Avenida de la Universidad 30, 28911 Legan\'es, Madrid, Spain\\
waltercarb@gmail.com, jomaro@math.uc3m.es\\
\\
Jos\'e Mar\'ia Sigarreta Almira\\
Facultad de Matem\'aticas, Universidad Aut\'onoma de Guerrero,\\
Carlos E. Adame No.54 Col. Garita, 39650 Acalpulco Gro., Mexico\\
jsmathguerrero@gmail.com\\
\\
Yadira Torres Nu\~nez\\
Departamento de Matem\'aticas, Universidad Carlos III de Madrid,\\
Avenida de la Universidad 30, 28911 Legan\'es, Madrid, Spain\\
ytnunez@math.uc3m.es}

\date{}

\maketitle{}

\begin{abstract}
The alliance polynomial of a graph $\Gamma$ with order $n$ and maximum degree $\delta_1$ is the polynomial $A(\Gamma; x) = \sum_{k=-\delta_1}^{\delta_1} A_{k}(\Gamma) \, x^{n+k}$, where $A_{k}(\Gamma)$ is the number of exact defensive $k$-alliances in $\Gamma$.
We provide an algorithm for computing the alliance polynomial.
Furthermore, we obtain some properties of $A(\Gamma; x)$ and its coefficients.
In particular, we prove that the path, cycle, complete and star graphs are characterized by their alliance polynomials.
We also show that the alliance polynomial characterizes many graphs that are not distinguished by other usual polynomials of graphs.
\end{abstract}



\
\section{Preliminaries.}

The study of the mathematical properties of alliances in graphs started in \cite{KHH}.
The defensive alliances in graphs is a topic of recent and increasing interest in
graph theory; see, for instance \cite{C,FLH3,H3,RVS1,RVGS,S,SBF1,SRV1,SRV2}.
The study of defensive alliances as a graph-theoretic concept has recently attracted a great deal of attention due to some interesting
applications in a variety of areas, including quantitative analysis of secondary RNA structures \cite{HKSZ} and national defense \cite{P}. Besides, defensive alliances are the mathematical model of web communities. Adopting the
definition of Web community proposed recently in \cite{FLG}, “a Web community is a set of web pages
having more hyperlinks (in either direction) to members of the set than to non-members”.

We begin by stating the used terminology.  Throughout
this paper, $\Gamma=(V,E)$ denotes a (not necessarily connected)
simple graph of order $|V|=n$ and size $|E|=m$.
We denote two adjacent vertices $u$ and $v$ by
$u\sim v$. For a nonempty set $X\subseteq V$, and a vertex $v\in V$,
 $N_X(v)$ denotes the set of neighbors that $v$ has in $X$:
$N_X(v):=\{u\in X: u\sim v\},$ and the degree of $v$ in $ X$ will be
denoted by $\delta_{X}(v)=|N_{X}(v)|.$ We denote the degree of a
vertex $v_i\in V$  by $\delta(v_i)=\delta_{\Gamma}(v_i)$ (or by $\delta_i$ for short) and
the degree sequence of $\Gamma$ by  $\{\delta_{1}, \delta_{2}, \ldots,\delta_{n}\}$ (ordered as follows $\delta_{1}\geq
\delta_{2}\geq \cdots \geq\delta_{n}$; then $\d_1$ is the maximum degree of $\Gamma$).
The subgraph induced by
$S\subset V$ will be denoted by  $\langle S\rangle $ and the
complement of the set $S\subset V$ will be denoted by $\bar{S}$.

A nonempty set $S\subseteq V$ is a \emph{defensive  $k$-alliance} in
$\Gamma=(V,E)$,  $k\in [-\delta_1,\delta_1]\cap \mathbb{Z}$, if for
every $ v\in S$,
\begin{equation}\label{cond-A-Defensiva} \delta _S(v)\ge \delta_{\bar{S}}(v)+k.\end{equation}

A vertex $v\in S$ is said to be $k$-\emph{satisfied} by the set $S$,
if \eqref{cond-A-Defensiva} holds. Notice that
\eqref{cond-A-Defensiva} is equivalent to
\begin{equation}\label{cond-A-Defensiva1} \delta (v)\ge 2\delta_{\bar{S}}(v)+k\end{equation}
and
\begin{equation}\label{cond-A-Defensiva2} 2\delta_S (v)\ge \delta(v)+k.\end{equation}


We consider the value of $k$ in the set of integers $\mathcal{K}:= [-\delta_1,\delta_1]\cap \mathbb{Z}$.
In some graphs $\G$, there are some values of  $k\in \mathcal{K}$, such that do not exist defensive
$k$-alliances in $\G$. For instance, for $k\ge 2$ in the star graph $S_n$, do no exist defensive $k$-alliances.
Besides, $V(\G)$ is a defensive $\delta_n$-alliance in $\G$.
Notice that for any $S$ there exists some $k\in\mathcal{K}$ such that it is a defensive $k$-alliance in $\G$.


Given $S\subseteq V$, we define
\begin{equation}\label{eq:k_exact}
    k_{S}:=\displaystyle\max_{} \{k\in\mathcal{K} \,:\,  S \text{ is a defensive $k$-alliance}\}.
\end{equation}

 We say that  $k_{S}$ is the \emph{exact index of alliance} of $S$, or also, $S$ is an \emph{exact defensive $k_{S}$-alliance} in $\G$, see e.g., \cite{C}.

\begin{proposition}\label{p:k_exact}
   Let $\G$ be a graph and let $S \subset V$. The following statements are equivalents:

   \begin{enumerate}
     \item {$k$ is the exact index of alliance of $S$.}

     \item {$S$ is a defensive $k$-alliance in $\G$ with one vertex $v\in S$ such that $\d_{S}(v) = \d_{\overline{S}}(v) + k$.}

     \item {$S$ is a defensive $k$-alliance but it is not a defensive $(k+1)$-alliance in $\G$.}
   \end{enumerate}

\end{proposition}

%
%
%
%
%

\begin{remark}\label{r:min_deg}
   The exact index of alliance of $S$ in $\G$ is
   \begin{equation}\label{eq:min_deg}
      k_{S}= \displaystyle\min_{v\in S} \{\d_{S}(v) - \d_{\overline{S}}(v)\}.
   \end{equation}

\end{remark}

Some parameters of a graph $\G$ allow to define polynomials on the graph $\G$, for instance,
the parameters associated to matching sets \cite{F,GG}, independent sets \cite{BDN,HL}, domination sets \cite{AAP,AP}, chromatic numbers \cite{R,T}, induced subgraphs \cite{TAM} and many others.
We choose the exact index of alliance in order to define the alliance polynomial of a graph (see Section 2).

A finite sequence of real numbers $(a_{0} , a_{1} , a_{2} , . . . , a_{n})$ is said to be \emph{unimodal}
if there is some $k \in \{0, 1, . . . , n\}$, called the \emph{mode} of the sequence, such that
$$a_{0} \leq . . . \leq a_{k-1} \leq a_{k} \quad \text{and} \quad a_{k} \geq a_{k+1} \geq . . . \geq a_{n};$$
the mode is unique if $a_{k-1} < a_{k}$ and $a_{k} > a_{k+1}$. A polynomial is called unimodal
if the sequence of its coefficients is unimodal.

In the next section, we introduce the alliance polynomial and obtain some of its properties.
In Section 3, we compute the alliance polynomial for some graphs and study its coefficients; in particular, we show that some of them are unimodal.
We investigate the alliance polynomials of path, cycle, complete and complete bipartite graphs. Also we prove that the path, cycle, complete and start graphs are characterized by their alliance polynomials.
Finally, in Section 4 we show that the alliance polynomial characterizes many graphs that are not distinguished by other usual polynomials of graphs.

\

\section{Alliance Polynomials.}

Let $\G$ be a graph with order $n$.
We define the \emph{alliance polynomial} of a graph $\G$ with variable $x$ as follows:

\begin{equation}\label{eq:Poly2}
    A(\G;x)= \displaystyle\sum_{S\subseteq V} \s_\G(S) \cdot x^{n+k_{S}},
\end{equation}
where $\s_\G(S)=1$ if $\langle S\rangle$ is nonempty and connected in $\G$, and $\s_\G(S)=0$ otherwise.

Other expression for this alliance polynomial is the following:

\begin{equation}\label{eq:Poly1}
    A(\G;x)= x^n \,\displaystyle\sum_{k\in \mathcal{K}} A_{k}(\G) x^{k}, \ \text{ with } A_{k}(\G) \text{ the number of connected exact defensive } k \text{-alliances in } \G.
\end{equation}


As an example, we compute now the alliance polynomial of the complete bipartite graph $K_{3,3}$.

Note that since $K_{3,3}$ is a cubic graph, we have $A_{k}(K_{3,3})=0$ for $k\in \{-2,0,2\}$.
In order to obtain $A(K_{3,3};x)$, we compute its non-zero coefficients.

\begin{description}
  \item[$A_{-3}(K_{3,3})=6$] {Since $K_{3,3}$ is cubic, we have that the number of exact defensive ($-3$)-alliances is $|V(K_{3,3})|=6$.}

  \item[$A_{-1}(K_{3,3})=33$] {We have that $S\subset V(K_{3,3})$ is an exact defensive ($-1$)-alliance, if both parts of $K_{3,3}$
  have some vertex in $S$ and one of them has just one vertex. Thus, we obtain from combinatorial arguments the result.}

  \item[$A_{1}(K_{3,3})=15$] {We have that $S\subset V(K_{3,3})$ is an exact defensive $1$-alliance, if $S \neq V(K_{3,3})$ and $S$ contains at lest two vertices of both parts of $K_{3,3}$. Thus, we obtain from combinatorial arguments the result.}

  \item[$A_{3}(K_{3,3})=1$] {Obviously, we have that the unique exact defensive $3$-alliance is the set of vertices of $K_{3,3}$.}
\end{description}

Then, we obtain
\[
A(K_{3,3};x) = 6x^3 + 33 x^5 + 15 x^7 + x^9.
\]

\medskip

The following procedure allows to compute the alliance polynomial of a graph $\G$ with order $n$.
Let $W=\{S_1,\dots,S_{2^n-1}\}$ be the collection of nonempty subsets of $V$.

%
%
%
%
%

\begin{algorithm}\label{algorithm}
${}$

Input: adjacency matrix of $\G$.

Output: alliance polynomial of $\G$.

\smallskip

The algorithm starts with $A(\G;x)=0$ and continues with the following steps, for $1 \le j \le 2^n-1$.
\begin{enumerate}
  \item {If $\langle S_j\rangle$ is a connected subgraph, then go to step (2), else replace $j$ by $j+1$ and apply this step again.}
  \item {Compute $k_{S_j}$.}
  \item {Add one term $x^{n+k_{S_j}}$ to $A(\G;x)$.}
  \item {Replace $j$ by $j+1$ and apply step (1) again.}
\end{enumerate}
\end{algorithm}

This algorithm for computing the alliance polynomial of a graph shows a complexity $O(m2^n$), furthermore, when it is running on $\D$-regular graphs its complexity is $O(n2^n)$. The algorithm looks for the $2^n-1$ nonempty induced subgraphs of $\G$. In step (1), for each induced subgraph, it analyzes if it is connected or not, using Depth-First Search (DFS) algorithm. It is a well known result that DFS algorithm complexity is $O(m)$, where $m$ is the number of edges of $\G$.
Furthermore, it is easy to check that step (2) has cost $O(n)$ and step (3) has cost $O(1)$.

\bigskip

An \emph{isomorphism of graphs} $\G_1$ and $\G_2$ is a bijection between the vertex sets of $\G_1$ and $\G_2$, $f: V(\G_1)\rightarrow V(\G_2)$ such that any two vertices $u$ and $v$ of $\G_1$ are adjacent in $\G_1$ if and only if $ƒ(u)$ and $ƒ(v)$ are adjacent in $\G_2$. If an isomorphism exists between $\G_1$ and $\G_2$, then the graphs are called \emph{isomorphic} and we write $\G_1 \simeq \G_2$.

\begin{remark}
  Let $\G_1$ and $\G_2$ be isomorphic graphs. Then $A(\G_1; x) = A(\G_2; x)$.
\end{remark}

\medskip

The following proposition shows general properties which satisfy the alliance polynomials.

\begin{proposition}\label{p:AlliPoly}
   Let $\G$ be a graph. Then, $A(\G;x)$ satisfies the following properties:

   \begin{enumerate}[i)]
     \item {All real zeros of $A(\G;x)$ are non-positive numbers.}

     \medskip

     \item {The value $0$ is a zero of $A(\G;x)$ with multiplicity $n - \d_1 \geq 1$.}

     \medskip

%

     \item {$\sum_{i=k}^{\d_1}A_{i}(\G)$ is the number of defensive $k$-alliances in $\G$ for every $k \in \mathcal{K}$.}

     \medskip

     \item {If $\G$ has at least an edge and its degree sequence has exactly $r$ different values $\{c_1,c_2,\ldots,c_r\}$,
     then $A(\G;x)$ has at least $r + 1$ terms: $x^{n-c_1},\ldots,x^{n-c_r},x^{n+\d_n}$.}

     \medskip

     \item { $A(\G;x)$ is a symmetric polynomial (either an even or an odd function)
     if and only if the degree sequence of $\G$ has either all values even or all odd.

     }
   \end{enumerate}
\end{proposition}

\begin{proof}
We prove separately each item.
\begin{enumerate}[i)]
  \item {Since the coefficients of $A(\G;x)$ are non-negatives, we have the result.}

  \medskip

  \item {Since $n+k\ge n-\d_1$ for any $k\in \mathcal{K}$, we have a common factor $x^{n-\d_1}$ in $A(\G;x)$ and $A_{-\d_1}(\G)\neq0$.}

%
%

  \medskip

  \item {If $S$ is an exact defensive $r$-alliance in $\G$ with $r\ge k$, then we have $\d_S(v) \geq \d_{\overline{S}}(v) + r \ge \d_{\overline{S}}(v) + k$  for all $v\in S$; in fact, $S$ is a defensive $k$-alliance in $\G$.
      This finishes the proof, since an exact defensive $r$-alliance in $\G$ with $r <k$ is not a defensive ($r+1$)-alliance and $r+1\le k$.
      }

  \medskip

  \item {Consider $v_1,v_2,\ldots,v_r \in V$ with $\d_{\G}(v_i) = c_i$ for all $i=1,\ldots,r$.
      Note that $\{v_i\}$ for $i=1,\ldots,r$ is an exact defensive ($-c_i$)-alliance, since $0 = \d_{S_i}(v_i) = \d_{\overline{S_i}}(v_i)- c_i = c_i - c_i$.
      Therefore, that makes appear the term $x^{n-c_i}$ in $A(\G;x)$ for all $i=1,\ldots,r$.
      Consider now a connected component $S$ of $\G$ and $u$ a vertex in $S$ with $\d_\G(u) = \d_n$. Hence, $S$ is an exact defensive $\d_n$-alliance in $G$, since we have
      \begin{equation}\label{eq:iii}
         \d_{S}(v) = \d_\G(v) \geq \d_{\overline{S}}(v) + \d_n = \d_n, \quad \forall v\in S
      \end{equation}
      and $\d_S(u)= \d_n$.
      So, that makes appear the term $x^{n+\d_n}$ in $A(\G;x)$.
      }

  \medskip

  \item {
      In order to prove the directed implication assume that $A(\G;x)$ is an even polynomials (the case odd is analogous). Let $c$ be any element of the degree sequence of $\G$ and $v\in V$ with $\d(v)=c$. By item v) we have $A_{-c}(\G)\neq 0$, then $n-c$ is even and $c \cong n (\text{mod } 2)$.
      So, we conclude that the elements in the degree sequence of $\G$ are either all even or all odd numbers.

      Finally, we prove the converse implication.
      Consider $S\subseteq V$ an exact defensive $k$-alliance.
      By Proposition \ref{p:k_exact}, there exists $v\in S$ with
      \[
      2\d_{S}(v) = \d_{\G}(v) + k.
      \]
      This finishes the proof since $\d_{\G}(v) + k$ is even.
      }
\end{enumerate}
\end{proof}

A \emph{cut vertex set} of a graph $\G = (V,E)$ is a subset $X \subsetneq V$ such that $\langle V\setminus X\rangle$ is a non-connected graph.

\begin{theorem}\label{th:Eval1}
Let $\G$ be any graph with order $n$. Then, we have the following statements
\begin{enumerate}
  \item {$A(\G;1) < 2^{n}$, and it is the number of connected induced subgraphs $\langle S\rangle$ in $\G$.}
  \item {The number of cut vertex sets of $\G$ is $2^n-1-A(\G;1)$.}
\end{enumerate}
\end{theorem}

\begin{proof}
By \eqref{eq:Poly2}, we have
\[
  A(\G;1)= \displaystyle\sum_{S\subset V}  \s_\G(S).
\]
Thus, $A(\G;1)$ is the number of connected induced subgraph $\langle S\rangle$ in $\G$; this amount is less that $2^{n}$, since we have $2^{n}-1$ nonempty subsets of $V$.

Let $c_k(\G)$ be the number of cut vertex sets of cardinality $k$ for $0\le k < n$ and
$s_k(\G)$ be the number of connected induced subgraphs of $\G$ with order $k$ for $0< k\le n$.
Note that $X$ is a cut vertex set if and only if $V(\G) \setminus X$ induces a non-connected subgraph.
Then, we have the following equality for every $0<k\le n$
$$c_{n-k}(\G) + s_k(\G) = {n \choose k}.$$
Finally, we obtain the result since $A(\G;1)=\sum_{k=1}^{n} s_k(\G)$.
\end{proof}

\smallskip

The following theorem shows some properties of coefficients and degree of alliance polynomial.

\begin{theorem}\label{t:properties}
   Let $A(\G;x)$ be the alliance polynomial of a graph $\G$ with $\Deg_{\min}(A(\G;x))$ and $\Deg(A(\G;x))$ the minimum degree and maximum degree of its terms, respectively.
   Then, $A(\G;x)$ satisfies the following statements:

   \begin{enumerate}[i)]
     \medskip

     \item {$\Deg_{\min}(A(\G;x)) = n - \d_1$ and its coefficient $A_{-\d_1}(\G)$ is the number of vertices in $\G$ with degree $\d_1$.
         }

     \medskip

     \item {$A_{-\d_1 + 1}(\G)$ is the number of vertices in $\G$ with degree $\d_1 - 1$.}

     \medskip

     \item {$A_{\d_n}(\G) > 0$.}

     \medskip

     \item {$n + \d_n \leq \Deg(A(\G;x)) \leq n + \d_1$.}

     \medskip

     \item {$A_{\d_1}(\G)$ is equal to the number of connected components in $\G$ which are $\d_1$-regular.}

     \medskip

     \item {There not exist defensive $k$-alliances in $\G$ for $k > \Deg(A(\G;x)) - n$.
     }

%
%
%
%
%
%
%

%

   \end{enumerate}
\end{theorem}

\begin{proof}
We prove separately each item.
\begin{enumerate}[i)]
  \item {The minimum value of $\mathcal{K}$ is $-\d_1$, so $\Deg_{\min}(A(\G;x)) \ge n - \d_1$.
      Consider now the sets $S_v=\{v\}$ with $\d_{\G}(v)=\d_1$, then $\langle S_v\rangle$ is connected and $S_v$ is an exact defensive ($-\d_1$)-alliance.
      Finally, it is clear that any $S\in V$ with more than one vertex is not an exact defensive ($-\d_1$)-alliance, since for any $v\in S$ we have
      \begin{equation}\label{eq:i}
      \d_{S}(v) - \d_{\overline{S}}(v) \geq  1 - (\d_1 - 1) > -\d_1 + 1.
      \end{equation}
      Then, $A_{-\d_1}(\G)$ is the number of vertices in $\G$ with degree $\d_1$.
      Note that, consequently, $A_{-\d_1}(\G) \le n$ and $A_{-\d_1}(\G) = n$ if and only if $\G$ is a regular graph.}

  \medskip

  \item {Similarly to the previous item, we consider the sets $S_v=\{v\}$ with $\d_{\G}(v)=\d_1 - 1$ and we obtain $A_{-\d_1 + 1}(\G) \geq NV_{\d_1 - 1}$ where $NV_{i}:= \{ \text{number of vertices in } \G \text{ with degree } i\}$; therefore, we obtain the equality since any $S\subset V$ with more than one vertex is an exact defensive $k$-alliance for $k \geq -\d_1 + 2$ by \eqref{eq:i}.
      }

  \medskip

  \item {This is a consequence of Proposition \ref{p:AlliPoly} iv).}

  \medskip

  \item {Item iii) gives the first inequality. The second one holds since $\d_1$ is the maximum value of $\mathcal{K}$.}

  \medskip

  \item {By \eqref{eq:Poly1}, $A_{\d_1}(\G)$ is the number of defensive $\d_1$-alliance in $\G$. 
      First, note that if $S$ is a defensive $\d_1$-alliance, then $S$ is an exact defensive $\d_1$-alliance since $\d_1$ is the maximum value in $\mathcal{K}$. Clearly, any connected component in $\G$ which is $\d_1$-regular is an exact defensive $\d_1$-alliance.

      Now, consider an exact defensive $\d_1$-alliance $S$ in $\G$. Hence, for any $v\in S$ we have
      \[
      \d_S(v) \geq \d_{\overline{S}}(v) + \d_1 \quad \Longrightarrow \quad \d_1 \geq \d_S(v) \geq \d_{\overline{S}}(v) + \d_1 \geq \d_1.
      \]
      Then, we have $\d_{S}(v) = \d_{\G}(v)= \d_1$ for every $v\in S$ and conclude that $S$ is a connected component in $\G$ which is $\d_1$-regular.
      }

  \medskip

  \item {Suppose that there is a defensive $k$-alliance $S$ in $\G$, in fact, $k_{S} \geq k$. Then, that makes appear the term $x^{n + k_{S}}$ in $A(\G;x)$ and so,
      \[
      n + k \leq n + k_{S} \leq \Deg(A(\G;x)).
      \]
      }
\end{enumerate}
\end{proof}

\begin{proposition}\label{c:ConComp}
   Let $\G$ be any connected graph. Then, $\G$ is regular if and only if
   \begin{equation}\label{eq:ConComp}
      A_{\d_1}(\G) = 1.
   \end{equation}
\end{proposition}

\begin{proof}
   If $\G$ is regular, then by Theorem \ref{t:properties} v) we obtain $A_{\d_1}(\G)=1$.
   Besides, if $A_{\d_1}(\G)=1$, then there is an exact defensive $\d_1$-alliance $S$ in $\G$ with $\d_S(v) \ge \d_{\bar{S}}(v) + \d_1 \ge \d_1$ for every $v\in S$ (i.e., $\d_S(v)=\d_1$ for every $v\in S$).
   So, the connectivity of $\G$ gives that $\G$ is a $\d_1$-regular graph.
\end{proof}

\begin{proposition}\label{l:subgraph}
   Let $\G$ be any graph and $G$ any proper subgraph of $\G$. Then
   \begin{equation*}
      A(\G;x) \neq A(G;x).
   \end{equation*}
\end{proposition}

\begin{proof}
   Since $G$ is a proper subgraph of $\G$, all connected induced subgraph of $G$ is a connected induced subgraph of $\G$ and at less one edge $e$ (with endpoints $u,v\in V$) of $\G$ is not contained in $G$.
   Hence, since $\langle \{u,v\}\rangle$ is connected in $\G$ but is no connected in $G$, we have $A(\G;1) > A(G;1)$ by Theorem \ref{th:Eval1}.
\end{proof}

The \emph{disjoint union} of graphs, sometimes referred to as simply \emph{graph union} is defined as follows. For two graphs $\G_1=(V_1,E_1)$ and $\G_2=(V_2,E_2)$ with disjoint vertex sets $V_1$ and $V_2$ (and hence disjoint edge sets), their union is the graph $\G_1\cup \G_2:=(V_1\cup V_2, E_1 \cup E_2)$. It is a commutative and associative operation.

\begin{theorem}\label{t:disjoints}
 Let $\G=\G_1 \cup \ldots \cup \G_r$ be the disjoint union of the graphs $\G_1,\ldots,\G_r$ $(r\ge2)$ with orders $n_1,\ldots,n_r$, respectively. Then we have
 \begin{equation}\label{eq:disjoints}
    A(\G;x) = x^{n - n_1} A(\G_1;x) + \ldots + x^{n-n_r} A(\G_r;x),
 \end{equation}
 where $n:=n_1+\ldots+n_r$.
\end{theorem}

\begin{proof}
   Since all connected induced subgraph of $\G$ is a connected induced subgraph of $\G_i$ for some $1\le i \le r$, and all exact defensive $k$-alliance in $\G$ is an exact defensive $k$-alliance in $\G_i$ for some $1\le i \le r$; we have that $\mathcal{K}(\G) = \bigcup_{i=1}^{r} \mathcal{K}(\G_i)$ and
   \[
   A_{k}(\G) = A_{k}(\G_1) + \ldots + A_{k}(\G_r), \quad \text{ for  } k\in \mathcal{K}(\G).
   \]

   So, we have
   \[
   A_{k}(\G) x^{n + k} = x^{n-n_1} A_{k}(\G_1) x^{n_1 + k} + \ldots + x^{n-n_r} A_{k}(\G_r) x^{n_r + k}, \quad \text{ for  } k\in \mathcal{K}(\G).
   \]
   Finally, if we sum in $k\in \mathcal{K}(\G)$, then we obtain the result.
\end{proof}


This result allows to obtain the alliance polynomial of the graph $\G\cup \{v\}$ obtained by adding to the graph $\G$ a single disjoint vertex $v$ (i.e., $v\notin V(\G)$). This operation is called \emph{vertex addition}.

\begin{corollary}\label{c:G+v}
Let $\G$ be any graph with order $n$ and let $v$ be a vertex such that $v\notin V(\G)$. Then
\[
A(\G\cup \{v\};x)= x\, A(\G;x)+ x^{n+1}.
\]
\end{corollary}

The $n$-vertex edgeless graph or \emph{empty graph} is the complement graph for the complete graph $K_n$; it is commonly denoted as $E_n$ for $n\ge1$.


\begin{corollary}\label{c:vacio}
 Let $n$ be a natural number with $n\ge1$.  If $A(\G; x) = n x^n$, then $\G$ is an isomorphic graph to $E_n$.
\end{corollary}

\begin{proof}
   Note that the empty graph $E_1$ satisfies $A(E_1;x) = x$.
   So, by Theorem \ref{t:disjoints} or Corollary \ref{c:G+v} we have that
   \[
   A(E_{n+1}) = x A(E_n;x) + x^{n+1}, \quad \forall n \geq 1.
   \]

   This implies that $A(E_n;x)= n x^n$.
   The uniqueness follows from items iii) and iv) in Theorem \ref{t:properties}.
\end{proof}

\begin{corollary}\label{c:G+Em}
Let $\G$ be any graph with order $n$. Then
\[
A(\G\cup E_m;x)= x^m\, A(\G;x)+ mx^{n+m}.
\]
\end{corollary}

\medskip

The \emph{graph join} $\G_1\uplus \G_2$ of two graphs is their graph union with all the edges that connect the vertices of the first graph $\G_1$ with the vertices of the second graph $\G_2$. It is a commutative operation.

\begin{theorem}\label{t:join}
  Let $\G_1,\G_2$ be two graphs with order $n_1$ and $n_2$, respectively. Then
  \[
  A(\G_1\uplus \G_2;x)= A(\G_1;x) + A(\G_2;x)+ \widetilde{A}(\G_1,\G_2;x),
  \]
  where $\widetilde{A}(\G_1,\G_2;x)$ is a polynomials with $\widetilde{A}(\G_1,\G_2;1)=(2^{n_1}-1)(2^{n_2}-1)$ and $Deg\big(\widetilde{A}(\G_1,\G_2;x)\big)=Deg\big(A(\G_1\cup \G_2;x)\big)$.
\end{theorem}

\begin{proof}
Let us define $\widetilde{A}(\G_1,\G_2;x)= A(\G_1\uplus \G_2;x) - A(\G_1;x) - A(\G_2;x)$.
First, if $S_1$ is a defensive alliance in $\G_1$ which provides a term $x^{n_1+k_{S_1}}$ in $A(\G_1;x)$, then $S_1$ provides a term $x^{n_1+n_2+k_{S_1}-n_2}=x^{n_1+k_{S_1}}$ in $A(\G_1\uplus \G_2;x)$.
It follows immediately that we obtain $A(\G_1;x)$ as an addend in $A(\G_1\uplus \G_2;x)$ when $S_1$ runs on the defensive alliances in $\G_1$.
Similarly, we obtain $A(\G_2;x)$ as an addend in $A(\G_1\uplus \G_2;x)$ when we consider the defensive alliances in $\G_2$.

In order to complete the summation in $A(\G_1\uplus \G_2;x)$ we consider $R_1 \subseteq V(G_1)$
(being either a defensive alliance in $\G_1$ or not) with $1\le r_1\le n_1$ elements and $R_2\subseteq V(\G_2)$ (being either a defensive alliance in $\G_2$ or not) with $1\le r_2\le n_2$ elements. Note that any $R_1\cup R_2$ is a defensive alliance in $\G_1\uplus \G_2$ since $\langle R_1\cup R_2\rangle$ is connected.
By Theorem \ref{th:Eval1}, we have
$$
\widetilde{A}(\G_1,\G_2;1)=\displaystyle\sum_{i=1}^{n_1} \sum_{j=1}^{n_2} {n_1 \choose i}{n_2 \choose j}=\left( \sum_{i=1}^{n_1} {n_1 \choose i} \right) \left( \sum_{j=1}^{n_2} {n_2 \choose j} \right) = (2^{n_1}-1)(2^{n_2}-1).
$$
However, the exact index of alliance of $R_1\cup R_2$ in $\G_1\uplus \G_2$ depends strongly on the particular geometry (topology) of $\G_1$ and $\G_2$.
In general, we can not determine the exact index of alliance of $R_1\cup R_2$ given its cardinality and degree sequence.

It is obvious that terms in $A(\G_1\uplus \G_2;x)$ provided from every $R_1\cup R_2$ with maximum degree are obtained from $R_1^*$ and $R_2^*$ defensive alliances with $\langle R_1^* \rangle,\langle R_2^* \rangle$ connected subgraphs and highest exact index of alliance in $\G_1$ and $\G_2$, respectively.
Hence, $$Deg\big(\widetilde{A}(\G_1,\G_2;x)\big)=n_1+n_2+\max\{k_{R_1^*},k_{R_2^*}\},$$
where the maximum is taken over all $R_1^*$, $R_2^*$ defensive alliances in $\G_1$, $\G_2$, respectively.
So, \eqref{eq:disjoints} finishes the proof since
\[
Deg\big( A(\G_1\cup \G_2;x) \big) = \max\left\{n_2 + Deg\big(A(\G_1;x)\big) , n_1 + Deg\big(A(\G_2;x)\big)\right\} = n_1 + n_2 + \max\{k_{R_1^*},k_{R_2^*}\},
\]
where the maximum is taken over all $R_1^*$, $R_2^*$ defensive alliances in $\G_1$, $\G_2$, respectively.
\end{proof}

Theorem \ref{t:join} allows to obtain the following result which will be useful (see Section \ref{ss:Polynomials for complete graphs}).
We denote by $\overline{\G}$ the \emph{complement graph} of $\G$ (note that $\overline{K}_n$ is isomorphic to the empty graph $E_n$).

\begin{theorem}\label{t:join2}
   Let $n,m$ be two positive integers. Then we have
   \begin{equation}\label{eq:join2}
     A(K_n\uplus \overline{K}_m;x)=A(K_n;x) \widetilde{A}_{m}(x) + m x^m
   \end{equation}
   where $\widetilde{A}_m(x)$ is a polynomial which just depend of $m$, in fact,
   \[
   \widetilde{A}_{m}(x)= \displaystyle\sum_{r=0}^{m} {m\choose r} x^{\min\{2r,m+1\}}.
   \]
\end{theorem}

\begin{proof}
First, we fix $S\subset V(K_n)$ with $1\le s\le n$ elements.
Note that $S$ provides a term $x^{2s-1}$ in $A(K_n;x)$. Consider $R\subset V(\overline{K}_m)$ with $0\le r\le m$ elements. Now we compute the exact index of alliance of $H_R=S\cup R$ in $K_n\uplus \overline{K}_m$.
We have
$$\d_{H_R}(v)-\d_{\overline{H_R}}(v)=(r+s-1)-(n-s + m-r)= 2s-1 -(n+m) + 2r, \quad \text{ for every } v\in S$$
and
$$\d_{H_R}(v)-\d_{\overline{H_R}}(v)=s-(n-s)= 2s -1 -(n+m)+ m+1, \quad \text{ for every } v\in R.$$
Then, $H_R$ provides a term $x^{2s-1 + \min\{2r,m+1\}}$ for each $R$.
Therefore, for each $S$ we obtain the polynomial $x^{2s-1} \cdot \widetilde{A}_m(x)$ when $R$ runs in the subsets of $V(\overline{K}_m)$.
In order to complete the sum, note that the defensive alliances without elements of $V(K_n)$ are just the set of single vertices of $V(\overline{K}_m)$.
Then \eqref{eq:Poly2} gives the result.
\end{proof}

Also, we can compute the alliance polynomials of $K_n\uplus K_{m}$ (see Proposition \ref{p:Kn}) and $\overline{K}_n\uplus \overline{K}_{m}$ (see Proposition \ref{p:Knm}).

\

\section{Characterization of path, cycle, complete and star graphs by its alliance polynomials}\label{Sec3}

In this section we obtain the explicit formulae for alliance polynomials of some classical classes of graphs using combinatorial arguments. We also study fundamental properties such as unimodality and the uniqueness of these polynomials.

Figure \ref{fig:1} shows two graphs $\G_1$ and $\G_2$ with the same order, size, degree sequence and number of induced subgraphs; however, these graphs have different alliance polynomials.
A simple computation gives $A(\G_1;x)=2x^7+4x^8+27x^9+50x^{10}+11x^{11}$ and $A(\G_2;x)=2x^7+4x^8+30x^9+47x^{10}+11x^{11}$. 


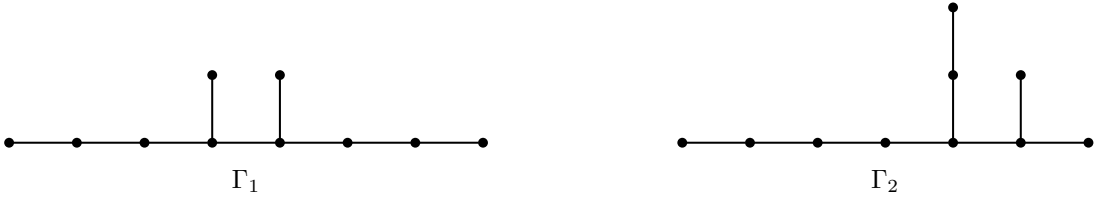
\begin{figure}[h]
\begin{minipage}[t]{8cm}
\centering
\scalebox{.9}
{\begin{pspicture}(-0.2,-0.2)(7.2,2.2)
\psline[linewidth=0.03cm,dotsize=0.07055555cm 2.5]{*-*}(0,0)(1,0)
\psline[linewidth=0.03cm,dotsize=0.07055555cm 2.5]{-*}(1,0)(2,0)
\psline[linewidth=0.03cm,dotsize=0.07055555cm 2.5]{-*}(2,0)(3,0)
\psline[linewidth=0.03cm,dotsize=0.07055555cm 2.5]{-*}(3,0)(4,0)
\psline[linewidth=0.03cm,dotsize=0.07055555cm 2.5]{-*}(4,0)(5,0)
\psline[linewidth=0.03cm,dotsize=0.07055555cm 2.5]{-*}(5,0)(6,0)
\psline[linewidth=0.03cm,dotsize=0.07055555cm 2.5]{-*}(6,0)(7,0)
\psline[linewidth=0.03cm,dotsize=0.07055555cm 2.5]{-*}(4,0)(4,1)
\psline[linewidth=0.03cm,dotsize=0.07055555cm 2.5]{-*}(3,0)(3,1)
\end{pspicture}}
\\
$\G_1$
\end{minipage}
\hfill
\begin{minipage}[t]{8cm}
\centering
\scalebox{.9}
{\begin{pspicture}(-0.2,-0.2)(6.2,2.2)
\psline[linewidth=0.03cm,dotsize=0.07055555cm 2.5]{*-*}(0,0)(1,0)
\psline[linewidth=0.03cm,dotsize=0.07055555cm 2.5]{-*}(1,0)(2,0)
\psline[linewidth=0.03cm,dotsize=0.07055555cm 2.5]{-*}(2,0)(3,0)
\psline[linewidth=0.03cm,dotsize=0.07055555cm 2.5]{-*}(3,0)(4,0)
\psline[linewidth=0.03cm,dotsize=0.07055555cm 2.5]{-*}(4,0)(5,0)
\psline[linewidth=0.03cm,dotsize=0.07055555cm 2.5]{-*}(5,0)(6,0)
\psline[linewidth=0.03cm,dotsize=0.07055555cm 2.5]{-*}(5,0)(5,1)
\psline[linewidth=0.03cm,dotsize=0.07055555cm 2.5]{-*}(4,0)(4,1)
\psline[linewidth=0.03cm,dotsize=0.07055555cm 2.5]{-*}(4,1)(4,2)
\end{pspicture}}
\\
$\G_2$
\end{minipage}
\caption{Graphs with same order, size, degree sequence and number of connected induced subgraphs such that $A(\G_1;x) \neq A(\G_2;x)$.} \label{fig:1}
\end{figure}

\
\subsection{Polynomials for path and cycle graphs}

\begin{proposition}\label{p:Pn}
   Let $P_n$ be a path graph with order $n\geq 2$. Then
   \begin{equation}\label{eq:Pn}
      A(P_n;x) = (n-2) \,x^{n-2} + 2 \,x^{n-1} + \frac{(n-2)(n+1)}2 \, x^n + x^{n+1}.
   \end{equation}

\end{proposition}

\begin{proof}
   We analyze the subsets with different cardinality separately.

   Let us consider any subset $S$ of $V(P_n)$ with connected induced subgraph $\langle S \rangle$, and $|S|=r$ with $r=1,\ldots, n$.

   If $r= 1$, then there are $n$ alliances.
                     \begin{itemize}
                       \item {Since there are two vertices with degree $1$, we have $2$ exact defensive ($-1$)-alliances. So, that makes appear the term
                           \[2 x^{n-1}.\]}
                       \item {Since there are $n-2$ vertices with degree $2$, we have $n-2$ exact defensive ($-2$)-alliances. So, that makes appear the term
                           \[(n-2)x^{n-2}.\]}
                     \end{itemize}

    Consider now the case $2\leq r\leq n-1$. The connectivity of $\langle S\rangle$ allows to compute $k_S$ since it is a sub-path with $r$ vertices.
    Then we have $n-r+1$ exact defensive $0$-alliances, since at least one endpoint of any induced $P_r$ attains the exact index of alliance $k_{P_r}=0$. So, we have the terms
         \[(n-r+1) x^n, \quad \text{ for every } 2\leq r\leq n-1.\]

    Finally, if $r= n$, then $S=V(P_n)$.
         We have just one exact defensive $1$-alliance, with the term
         \[x^{n+1}.\]

   Then, we obtain
   \[
   \begin{aligned}
   A(P_n;x) &= (n-2) \,x^{n-2} + 2 \,x^{n-1} + \displaystyle\sum_{r=2}^{n-1}(n-r+1) \, x^n + x^{n+1},\\
   &= (n-2) \,x^{n-2} + 2 \,x^{n-1} + \frac{(n-2)(n+1)}2 \, x^n + x^{n+1}.
   \end{aligned}
   \]
%
\end{proof}

We have the following consequences of Proposition \ref{p:Pn}.

\begin{corollary}\label{c:Pn}
   Let $P_n$ be the path graph with $n$ vertices. Then $A(P_n;x)$ is unimodal if and only if $2 \le n \le 4$.
\end{corollary}

\begin{proof}
By simple computation we can check that $A(P_n;x)$ is unimodal for $2 \le n \le 4$, since
$A(P_2;x)=2x+x^3$, $A(P_3;x)=x+2x^2+2x^3+x^4$ and $A(P_4;x)=2x^2+2x^3+5x^4+x^5$.
But, for $n > 4$ we have that $A_{-2}(P_n) = n-2 > 2 = A_{-1}(P_n) < (n-2)(n+1)/2 = A_{0}(P_n)$.
\end{proof}

Now we characterize graphs $\G$ with $A(\G; x) = A(P_t ; x)$.


\begin{theorem}\label{t:UniqPn}
  Let $t$ be a natural number with $t\ge2$.  If $A(\G; x) = A(P_t ; x)$, then $\G$ is an isomorphic graph to $P_t$.
\end{theorem}

\begin{proof}
   Let us consider a graph $\G$ with $A(\G;x)=A(P_t;x)$; denote by $n$ the order of $\G$ and by $\D_\G$ the maximum degree of $\G$.

   Assume first that $t\ge 3$.
   By items i) and ii) in Theorem \ref{t:properties},
   $n-\D_\G=t-2$,
   $\G$ has $t-2$ vertices of degree $\D_\G$, and $2$ vertices of degree $\D_\G-1$.
   So, we have $n\ge t$.

   Assume now that $t=2$.
   Then $A(\G;x)=A(P_2;x)=2x+x^3$.
   By Theorem \ref{t:properties} i),
   $n-\D_\G=1$ and
   $\G$ has $2$ vertices of maximum degree $\D_\G$.
   So, we have $n\ge t$.

   Hence, $n\ge t$ for every $t \ge 2$.

   By Theorem \ref{t:properties} iv), we have $t+1 \geq n + \d_{\G}$ where $\d_\G$ is the minimum degree of $\G$.
   So, $\d_{\G}$ is either $0$ or $1$. Hence, if $n > t$, then $n=t+1$ and $\d_{\G}=0$.
   Besides, the maximum degree of $A(\G;x)$ is greater than $t+1$ since $\G$ has a connected component with vertex of positive degree.
   This is a contradiction, thus $n=t$ and then $t-\D_\G=t-2$ if $t\ge 3$, and $2-\D_\G=1$ if $t=2$;
   therefore, $\D_\G=2$ if $t\ge 3$, and $\D_\G=1$ if $t=2$.

   Hence, if $t=2$, $\G$ is an isomorphic graph to $P_2$.
   If $t\ge 3$, then $\G$ has $t-2$ vertices of degree $2$ and $2$ vertices of degree $1$.
   If $\G$ is disconnected, then $A(\G;x)$ has at least two terms $x^k$ with $k>t$, one for each connected component.
   But this is a contradiction since $A(\G;x)=A(P_t;x)$.
   So, $\G$ is connected and this implies that $\G$ is an isomorphic graph to $P_t$.
\end{proof}

\begin{proposition}\label{p:Cn}
   Let $C_n$ be a cycle graph with order $n\geq 3$. Then
   \begin{equation}\label{eq:Cn}
      A(C_n;x) = n \, x^{n-2} + n(n-2) \, x^n + x^{n+2}.
   \end{equation}

\end{proposition}

\begin{proof}
   We analyze the subsets with different cardinality separately.

   Let us consider any subset $S$ of $V(C_n)$ with connected induced subgraph $\langle S \rangle$, and $|S|=r$ with $r=1,\ldots, n$.

   If $r= 1$, then we have $n$ exact defensive ($-2$)-alliances. So, that makes appear the term
                           \[n x^{n-2}.\]
   Consider now the case $2\leq r\leq n-1$. The connectivity of $\langle S\rangle$ allows to compute $k_S$ since it is a path with $r$ vertices.
   Then we have $n$ exact defensive $0$-alliances, since the end vertices of the induced $P_r$ attain the exact index of alliance $k_{P_r}=0$.
   So, we have the term
         \[n x^n, \quad \text{ for every  } 2\le r \le n-1.\]

    Finally, if $r= n$, then $S=V(C_n)$.
    We have an exact defensive $2$-alliance with the term
         \[x^{n+2}.\]

   Then, we obtain $A(C_n;x) = n \,x^{n-2} + n(n-2) \, x^n + x^{n+2}$.
\end{proof}

\begin{corollary}\label{c:C_n}
   Let $C_n$ be a cycle graph with order $n \ge 3$. Then $A(C_n;x)$ is unimodal.
\end{corollary}

Here we want to characterize graphs $\G$ with $A(\G; x) = A(C_t ; x)$.


\begin{theorem}\label{t:UniqCn}
  Let $t$ be a natural number with $t\ge3$.  If $A(\G; x) = A(C_t ; x)$, then $\G$ is an isomorphic graph to $C_t$.
\end{theorem}

\begin{proof}
   Let us consider a graph $\G$ with order $n$ such that $A(\G;x)=A(C_t;x)$; denote by $\D_\G$ the maximum degree of $\G$ and by $\d_\G$ its minimum degree.
   By Theorem \ref{t:properties} i), $\G$ has $t$ vertices of degree $\D_\G$, so $n \geq t$. Besides, $n + \d_{\G} \leq t+2 \leq n + \D_\G$. Hence, $\d_\G \leq 2$.

    Assume that $n > t$. Then $\d_\G$ is either $0$ or $1$.

    If $\d_\G = 0$, then Proposition \ref{p:AlliPoly} iv) makes appear the term $x^{n}$. Since $x^{t+1}$ does not appear in $A(C_t;x)$, we obtain $n \geq t+2$.
    Furthermore, it appears one term, associated to one connected component with vertices of positive degree, with exponent $n + \D_\G > n$, but this is impossible since $A(C_t;x)$ has degree $t+2$.

    Hence $\d_\G = 1$ and $n = t + 1$. So, by Theorem \ref{t:properties} i), $\G$ has $t$ vertices of degree $\D_\G=3$ and one vertex of degree $1$.
     Denote by $v$ the vertex of $\G$ with degree $1$ and by $S$ the connected component of $\G$ containing $v$.
     Clearly, $S$ is an exact defensive $1$-alliance in $\G$, and then the term $x^{(t+1) + 1}$ appears in $A(\G;x)$; but $S\setminus\{v\}$ is an exact defensive $1$-alliance in $\G$.
     This is a contradiction since there is just one term $x^{t+2}$ in $A(\G;x)$.

    Hence, we have $n=t$. Besides, by Theorem \ref{t:properties} i), $\G$ is a regular graph and $\D_{\G} = 2$. Since $A(C_t;x)$ is a monic polynomial with degree $t+2$, the number of connected components of $\G$ is $1$ by Theorem \ref{t:properties} v), and so, $\G$ is connected.
\end{proof}

\subsection{Polynomials for complete graphs}
\label{ss:Polynomials for complete graphs}

Since $K_{n+1}$ is an isomorphic graph to $K_n\uplus \overline{K}_1$ for every $n\ge1$, Theorem \ref{t:join2} has the following consequences.

\begin{proposition}\label{p:Kn}
   Let $K_n$ be a complete graph with order $n\geq 1$. Then
   \begin{equation}\label{eq:Kn}
      A(K_n;x) = \frac{(x^2 +1)^n - 1}x.
   \end{equation}
\end{proposition}

%
%
%
%
%


\begin{corollary}\label{c:K_n}
   Let $K_n$ be the complete graph with order $n$. Then $A(K_n;x)$ is unimodal.
\end{corollary}

Now we characterize graphs $\G$ with $A(\G; x) = A(K_t ; x)$.


\begin{theorem}\label{t:UniqKn}
  If $A(\G; x) = A(K_t ; x)$, then $\G$ is an isomorphic graph to $K_t$.
\end{theorem}

\begin{proof}
   Consider a graph $\G$ with order $n$ such that $A(\G;x)=A(K_t;x)$. By Theorem \ref{t:properties} i), $\G$ has $t$ vertices of maximum degree $\D_\G = n-1$, so $n \geq t$.
   Denote by $v_1,v_2,\ldots,v_t$ the vertices of $\G$ with maximum degree $n - 1$. Hence, we have that $\G$ contains a clique $\langle\{v_1,v_2,\ldots,v_t\}\rangle$ isomorphic to $K_t$.
   If $n > t$, then Proposition \ref{l:subgraph} gives $A(\G;x) \neq A(K_t;x)$. So, we obtain that $n = t$.
   Finally, since $n=t$, $\G$ is an ($t-1$)-regular graph. Therefore, $\G$ is an isomorphic graph to $K_t$.
\end{proof}

\medskip

Since a complete graph without one of its edges $K_n/e$ is isomorphic to $K_{n-2}\uplus \overline{K}_2$ for every $n\ge3$, Theorem \ref{t:join2} has the following consequence.

\begin{proposition}\label{t:Kn-e}
   Let $K_n/e$ be a complete graph without one of its edges, with $n \geq 2$ vertices. Then,
%
%
   \begin{equation}\label{eq:kn-e}
      A(K_n/e;x) = \frac{(x^2 + 1)^n - (x^4 - x^3)(x^2 + 1)^{n-2} + x^3 - 2x^2 -1}x.
   \end{equation}
\end{proposition}

Proposition \ref{t:Kn-e} gives the following results.

\begin{corollary}\label{c:K_n-e}
   Let $K_n/e$ be the complete graph with $n \geq 2$ vertices, without one of its edges. Then $A(K_n/e;x)$ is unimodal if and only if $2 \le n \le 4$.
\end{corollary}

\begin{proof}
   We can check that $A(K_n/e;x)$ is unimodal for $2 \le n \le 4$, since
$A(K_2/e;x)=A(E_2;x)=2x^2$, $A(K_3/e;x)=A(P_3;x)=x+2x^2+2x^3+x^4$ and $A(K_4/e;x)=2x+2x^2+5x^3+2x^4+2x^5+x^6$.
But, for $n > 4$ we have that $A_{-(n-1)}(K_n/e) = n-2 > 2 = A_{-(n-2)}(K_n/e) < {n \choose 2} - 1 = A_{-n+3}(K_n/e)$.
\end{proof}

Now we characterize graphs $\G$ with $A(\G; x) = A(K_t/e ; x)$.


\begin{theorem}\label{c:Kn-e}
   Let $t$ be a natural number with $t\ge2$.  If $A(\G; x) = A(K_t/e ; x)$, then $\G$ is an isomorphic graph to $K_t/e$.
\end{theorem}

\begin{proof}
   If $t=2$, then the result follows from Corollary \ref{c:vacio}.
   Assume now that $t\ge 3$.

   Let us consider a graph $\G$ with order $n$ such that $A(\G;x)=A(K_t/e;x)$.
   By items i) and ii) in Theorem \ref{t:properties}, $\G$ has $t-2$ vertices of maximum degree $\D_\G = n-1$ and $2$ vertices of degree $n - 2$, so $n \geq t$.
   Denote by $v_1,\ldots,v_{t-2}$ the vertices of $\G$ with maximum degree $n - 1$ and by $w_1,w_2$ the vertices with degree $n - 2$.
   Hence, we have that $\G$ contains a subgraph $\langle\{v_1,\ldots,v_{t-2},w_1,w_2\}\rangle$ which is either a clique or an isomorphic graph to $K_t/e$,
   depending on whether or not $w_1$ is adjacent to $w_2$ in $\G$.
   If $n > t$, then Proposition \ref{l:subgraph} gives $A(\G;x) \neq A(K_t;x)$. So, we obtain that $n = t$.

   Note that any nonempty subset $S$ of $V(\G)$ induces a connected subgraph $\langle S\rangle$ of $\G$, if $S\neq\{w_1,w_2\}$. Obviously, $A(\G;1)= 2^t-2$ and this is a characterization of the graph $K_t/e$, since a graph with one more induced connected subgraph is isomorphic to $K_t$.
   Furthermore, any graph $\G$ with order $t$ obtained from $K_t$ by removing at least two edges, does not satisfy the condition $A(\G;1)= 2^t-2$.
   Since $A(\G;x)=A(K_t/e;x)$ and $\G$ has order $t$, then $\G$ is isomorphic to $K_t/e$.
\end{proof}

\subsection{Polynomials for completed bipartite graphs}

Since $\overline{K}_n\uplus \overline{K}_m = K_{n,m}$,
an argument similar to the ones in the proofs of Theorems \ref{t:join} and \ref{t:join2} allows to obtain $A(\overline{K}_n\uplus \overline{K}_m;x)$.

\begin{proposition}\label{p:Knm}
   Let $K_{n,m}$ be a complete bipartite graph with $n,m \ge 1$. Then
   \begin{equation}\label{eq:Knm}
      A(K_{n,m};x) = n\, x^n + m\, x^m + \displaystyle\sum_{k=2}^{n+m} \quad \sum_{i,j>0\,,\, i+j = k} {n \choose i}{m \choose j} x^{n + m + \min\{2i - n , 2j - m\}}.
   \end{equation}
\end{proposition}

\begin{proof}

   Fix $n \geq 1$ and $m\geq 1$. Let us consider any subset $S$ of $V(K_{n,m})$ with connected induced subgraph $\langle S \rangle$ and $|S|=k$ with $k=1,\ldots, n+m$.

    If $k= 1$, then there are $n+m$ alliances.
                     \begin{itemize}
                       \item {If $S$ is a vertex associated to $n$, we have $n$ exact defensive ($-m$)-alliances. So, that makes appear the term
                           \[n x^{n+m-m}.\]}
                       \item {If $S$ is a vertex associated to $m$, we have $m$ exact defensive ($-n$)-alliances. So, that makes appear the term
                           \[m x^{n+m-n}.\]
                           }
                     \end{itemize}

    Consider now the case $2\leq k\leq n+m$.
    Obviously, any subset $S$ of $V(G)$ with $k\geq 2$ elements induces a connected subgraph of $K_{n,m}$, if and only if it contains elements in both parts.
    Then, we have ${n \choose i}{m \choose j}$ exact defensive $\min\{j - (m-j),i - (n-i)\}$-alliances for each couple $i,j >0$ such that $i+j=k$
    (by choosing $i$ vertices associated to $n$ and $j$ vertices associated to $m$).

    So, we have the terms
    \[\displaystyle\sum_{i,j >0,\,i+j=k}{n \choose i}{m \choose j} \, x^{n + m + \min\{2j-m,2i-n\}}.\]

   Then, we obtain
      \[
      A(K_{n,m};x) = n\, x^n + m\, x^m + \displaystyle\sum_{k=2}^{n+m} \quad \sum_{i,j>0\,,\, i+j = k} {n \choose i}{m \choose j} x^{n + m + \min\{2i - n , 2j - m\}}.
      \]
\end{proof}

\medskip

The complete bipartite graph $K_{n-1,1}$ is called an $n$ \emph{star graph} $S_n$. We have the following consequence of Theorem \ref{t:join2} (since $S_n$ is an isomorphic graph to $K_1\uplus \overline{K}_{n-1}$ for every $n\ge2$) or Proposition \ref{p:Knm}.

\begin{corollary}\label{c:Sn}
   Let $S_n$ be star graph with order $n\geq 2$. Then
   \begin{equation}\label{eq:Sn}
      A(S_n;x) = A(K_{n-1,1};x) = \displaystyle\sum_{k=0}^{\lfloor (n-1)/2\rfloor} {n-1 \choose k} x^{2k+1} + (n-1) x^{n-1} + x^{n + 1} \sum_{k=\lceil n/2\rceil}^{n-1} {n-1 \choose k}.
   \end{equation}
\end{corollary}

Here we want to characterize graphs $\G$ with $A(\G; x) = A(S_t ; x)$.


\begin{theorem}\label{t:UniqSn}
  Let $t$ be a natural number with $t\ge2$.  If $A(\G; x) = A(S_t ; x)$, then $\G$ is an isomorphic graph to $S_t$.
\end{theorem}

\begin{proof}
   If $t=2$ then Theorem \ref{t:UniqPn} gives the result. Fix $t\ge 3$.

   Let us consider a graph $\G$ with order $n$ such that $A(\G;x) = A(S_t;x)$. Since $\Deg_{min}(A(\G;x)) = 1$, there is $v\in V(\G)$ such that $v \sim w$ for all $w\in V(\G) \setminus\{v\}$. Therefore, $\G$ is a connected graph, $\d_\G$ (the minimum degree of $\G$) is greater that $0$ and $\G$ contains an isomorphic subgraph $\G_S$ of $S_{n}$.
   Hence, any $S \subseteq V(\G)$ which induces a connected subgraph $\langle S\rangle$ in $\G_S$, induces a connected subgraph in $\G$, too. So,
   \begin{equation}\label{eq:UniqSn1}
      A(\G;1) \geq A(\G_S;1) = A(S_{n};1).
   \end{equation}

   Since $\Deg(A(\G;x)) = t + 1$, we have $n + \d_\G \leq t+1$, and so, $n \leq t$.
   But, by \eqref{eq:Sn}, we have
   $$2^{n} > A(\G;1) = t-1 + \sum_{k=0}^{t-1} {t-1 \choose k} = 2^{t-1} + t-1 > 2^{t-1},$$
   and this condition implies that $n \geq t$.
   Thus, $n = t$.

   Seeking for a contradiction assume that there are $w_1,w_2 \in V(\G)\setminus\{v\}$ such that $w_1 \sim w_2$.
   Then, $\{w_1 , w_2\}$ induces a connected subgraph in $\G$, but not in $\G_S$; and so,
   \[A(\G;1) > A(S_t;1) \quad \Longrightarrow \quad A(\G;x) \neq A(S_t;x).\]
   This is the contradiction we were looking for, and so, $\G$ is isomorphic to $S_t$.
\end{proof}

\section{Distinctive power of alliance polynomial}
In this section we explain the distinctive power of the alliance polynomial of a graph. This is an interesting difference with others well-known polynomials of graphs.

We denote by $D(\G;x)$ the domination polynomial of $\G$ (see \cite{AAP}), by $I(\G;x)$ the independence polynomial of $\G$ (see \cite{HL}), by $m(\G;x)$ the matching polynomial (see \cite{F}), by $p(\G;x)$ the characteristic polynomial, by $T(\G;x,y)$ the Tutte polynomial (see \cite{T}), by $P(\G;x,y)$ the bivariate chromatic polynomial introduced in \cite{DPT}, and by $Q(\G;x,y)$ the subgraph component polynomial introduced in \cite{TAM}.

We say that a graph $\G$ is characterized by a graph polynomial $f$ if for every graph $\G'$ such that $f(\G') = f(\G)$ we have that $\G'$ is isomorphic to $\G$. The class of graphs $K$ is characterized by a graph polynomial $f$ if every graph $\G \in K$ is characterized by $f$.

This notion has been studied in \cite{MN,N}, for the chromatic polynomial, the Tutte polynomial and the matching polynomial. It
is shown, e.g., that several well-known families of graphs are determined by their Tutte polynomial, among them the class of wheels, squares of cycles, complete multipartite graphs, ladders, Möbius ladders, and hypercubes.
In Section 3, we have proved that path, cycle, complete and star graphs are characterized by their alliance polynomials.
In \cite{RST} the authors prove that the family of alliance polynomials of cubic graphs is a special one, since it does not contain alliance polynomials of graphs which are not cubic; and they also prove that the cubic graphs with at most $10$ vertices are characterized by their alliance polynomials.
Furthermore, in \cite{CRST2} the authors prove a similar result for the family of alliance polynomials of $\D$-regular connected graphs with $\D\le 5$, i.e., it does not contain alliance polynomials of graphs which are not connected $\D$-regular.

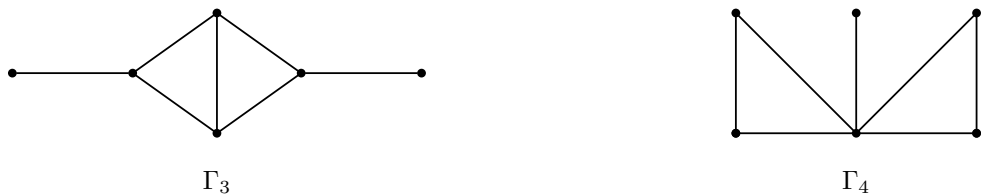
\begin{figure}[h]
\begin{minipage}[t]{8cm}
\centering
\scalebox{.8}
{\begin{pspicture}(-4,-0.4)(4,2.4)
\psline[linewidth=0.03cm,dotsize=0.07055555cm 2.5]{*-*}(-3.4,1)(-1.4,1)
\psline[linewidth=0.03cm,dotsize=0.07055555cm 2.5]{*-*}(3.4,1)(1.4,1)
\psline[linewidth=0.03cm,dotsize=0.07055555cm 2.5]{*-*}(0,0)(0,2)
\psline[linewidth=0.03cm,dotsize=0.07055555cm 2.5]{-}(-1.4,1)(0,0)(1.4,1)(0,2)(-1.4,1)
\end{pspicture}}
\\
$\G_3$
\end{minipage}
\hfill
\begin{minipage}[t]{8cm}
\centering
\scalebox{.8}
{\begin{pspicture}(-2.4,-0.4)(2.4,2.4)
\psline[linewidth=0.03cm,dotsize=0.07055555cm 2.5]{*-*}(-2,0)(-2,2)
\psline[linewidth=0.03cm,dotsize=0.07055555cm 2.5]{*-*}(2,0)(2,2)
\psline[linewidth=0.03cm,dotsize=0.07055555cm 2.5]{*-*}(0,0)(0,2)
\psline[linewidth=0.03cm,dotsize=0.07055555cm 2.5]{-*}(-2,2)(0,0)(-2,0)
\psline[linewidth=0.03cm,dotsize=0.07055555cm 2.5]{-*}(2,2)(0,0)(2,0)
\end{pspicture}}
\\
$\G_4$
\end{minipage}
\caption{Graphs with same characteristic polynomial.} \label{fig:12}
\end{figure}

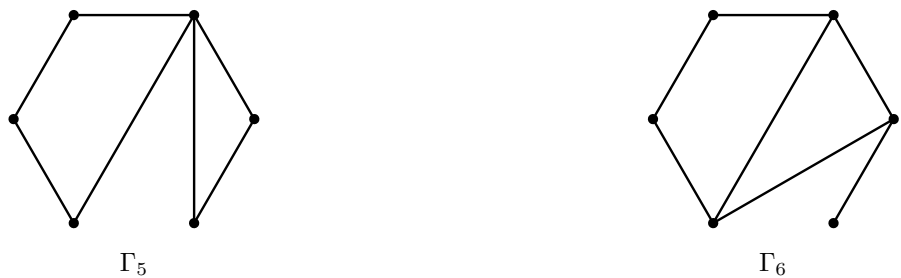
\begin{figure}[h]
\begin{minipage}[t]{8cm}
\centering
\scalebox{.8}
{\begin{pspicture}(-2.2,-2)(2.2,2)
\psline[linewidth=0.04cm,dotsize=0.07055555cm 2.5]{*-*}(1,-1.73)(2,0)
\psline[linewidth=0.04cm,dotsize=0.07055555cm 2.5]{-*}(2,0)(1,1.73)
\psline[linewidth=0.04cm,dotsize=0.07055555cm 2.5]{-*}(1,1.73)(-1,1.73)
\psline[linewidth=0.04cm,dotsize=0.07055555cm 2.5]{-*}(-1,1.73)(-2,0)
\psline[linewidth=0.04cm,dotsize=0.07055555cm 2.5]{-*}(-2,0)(-1,-1.73)
\psline[linewidth=0.04cm,dotsize=0.07055555cm 2.5]{-}(-1,-1.73)(1,1.73)(1,-1.73)
\end{pspicture}}
\\
$\G_5$
\end{minipage}
\hfill
\begin{minipage}[t]{8cm}
\centering
\scalebox{.8}
{\begin{pspicture}(-2.2,-2)(2.2,2)
\psline[linewidth=0.04cm,dotsize=0.07055555cm 2.5]{*-*}(1,-1.73)(2,0)
\psline[linewidth=0.04cm,dotsize=0.07055555cm 2.5]{-*}(2,0)(1,1.73)
\psline[linewidth=0.04cm,dotsize=0.07055555cm 2.5]{-*}(1,1.73)(-1,1.73)
\psline[linewidth=0.04cm,dotsize=0.07055555cm 2.5]{-*}(-1,1.73)(-2,0)
\psline[linewidth=0.04cm,dotsize=0.07055555cm 2.5]{-*}(-2,0)(-1,-1.73)
\psline[linewidth=0.04cm,dotsize=0.07055555cm 2.5]{-}(2,0)(-1,-1.73)(1,1.73)
\end{pspicture}}
\\
$\G_6$
\end{minipage}
\caption{Graphs with same bivariate chromatic polynomial.} \label{fig:34}
\end{figure}


We denote by $G_1\Box G_2$ and $G_1\boxtimes G_2$ the Cartesian and the strong products of $G_1$ and $G_2$, respectively.

\begin{proposition}\label{p:distinct}
For the graphs $\G_i$, $i = 1,...,6$, from Figures \ref{fig:1}, \ref{fig:12}, \ref{fig:34} and for $P_4$, $K_{1,3}$, $P_5$, $P_2\cup C_3$, $K_{3,3}$, $P_2\Box C_3$, $P_2\boxtimes P_3$ and $E_2\uplus P_4$ we have

\begin{enumerate}[$(1)$]
  \item {$p(\G_3;x) = p(\G_4;x)$ but $A(\G_3;x) \neq A(\G_4;x)$.}
  \item {$m(P_2\cup C_3;x) = m(P_5;x)$ but $A(P_2\cup C_3;x) \neq A(P_5;x)$.}
  \item {$I(P_2\boxtimes P_3;x)=I(E_2\uplus P_4;x)$ but $A(P_2\boxtimes P_3;x) \neq A(E_2\uplus P_4;x)$.}
  \item {$D(K_{3,3};x) = D(P_2\Box C_3;x)$ but $A(K_{3,3};x) \neq A(P_2\Box C_3;x)$.}
  \item {$P(\G_5;x,y) = P(\G_6;x,y)$ but $A(\G_5;x) \neq A(\G_6;x)$.}
  \item {$T(P_4;x,y) = T(K_{1,3};x,y)$ but $A(P_4;x) \neq A(K_{1,3};x)$.}
  \item {$Q(\G_1;x,y) = Q(\G_{2};x,y)$ but $A(\G_1;x) \neq A(\G_2;x)$.}
\end{enumerate}
\end{proposition}

\begin{proof}
Proposition \ref{p:AlliPoly} v) gives that $A(\G_3;x)$, $A(P_2\boxtimes P_3;x)$ and $A(\G_5;x)$ are symmetric polynomials, but $A(\G_4;x)$, $A(E_2\uplus P_4;x)$ and $A(\G_6;x)$ are not symmetric; then $A(\G_3;x)\neq A(\G_4;x)$, $A(P_2\boxtimes P_3;x) \neq A(E_2\uplus P_4;x)$ and $A(\G_5;x)\neq A(\G_6;x)$.
Besides, by Theorem \ref{t:UniqPn} we have that $P_4$ and $P_5$ are characterized by their alliance polynomials, and so, $A(P_2\cup C_3;x) \neq A(P_5;x)$ and $A(P_4;x) \neq A(K_{1,3};x)$. Furthermore, by \cite[Proposition 3.1]{RST} we have $A(K_{3,3};x) \neq A(P_2\Box C_3;x)$.
Besides, $A(\G_1;x)\neq A(\G_2;x)$ (see the beginning of Section \ref{Sec3}).
A simple computation gives $p(\G_3;x) = p(\G_4;x)$, $m(P_2\cup C_3;x) = m(P_5;x)$, $I(P_2\boxtimes P_3;x)=I(E_2\uplus P_4;x)$ and $D(K_{3,3};x) = D(P_2\Box C_3;x)$. So, items (1), (2), (3) and (4) hold.
Item (5) follows from \cite{DPT}.
Since Tutte polynomial does not distinguish trees of the same size, we deduce item (6).
Finally, $Q(\G_1;x,y) = Q(\G_{2};x,y)$ follows from \cite{TAM}, and we have item (7).
\end{proof}

The results in Section 3 and \cite{CRST2,RST} suggest the conjecture that every graph can be characterized by its alliance polynomial, although it seems hard to be proved.

However, if our conjecture turned out to be false, we think that the study of the following problem could be of interest.

\begin{problem}\label{problema}
Are there simple graphs distinguished by $p(\G;x)$, $m(\G;x)$, $I(\G;x)$, $D(\G;x)$, $P(\G;x,y)$, $T(\G;x,y)$ or $Q(\G;x,y)$ which are not distinguished by $A(\G;x)$?
\end{problem}

\section*{Acknowledgements}

This work was partly supported by a grant for Mobility of own research program at the University Carlos III de Madrid and a grant from CONACYT (CONACYT-UAG I0110/62/10), M\'exico.

\end{document}